\documentclass{amsart}
\usepackage{amsmath,amsthm,amssymb}
\usepackage{graphicx}

\newtheorem{theorem}{Theorem}[section]

\newtheorem{lemma}[theorem]{Lemma}
\newtheorem{corollary}[theorem]{Corollary}

\newtheorem{conjecture}[theorem]{Conjecture}

\theoremstyle{definition}
\newtheorem{definition}[theorem]{Definition}

\theoremstyle{remark}
\newtheorem{remark}[theorem]{Remark}

\theoremstyle{theorem}

\title[The simple type conjecture for mod 2 Seiberg--Witten invariants]{The simple type conjecture for mod 2 Seiberg--Witten invariants}
\author[Tsuyoshi Kato]{Tsuyoshi Kato}
\address{Department of Mathematics, Kyoto University, Kyoto, 606-8502, Japan}
\email{tkato@math.kyoto-u.ac.jp}

\author[Nobuhiro Nakamura]{Nobuhiro Nakamura}
\address{Department of Mathematics, Osaka Medical College, 2-7 Daigaku-machi, Takatsuki City, Osaka, 569-8686, Japan}
\email{mat002@osaka-med.ac.jp}

\author[Kouichi Yasui]{Kouichi Yasui}
\address{Department of Pure and Applied Mathematics, Graduate School of Information Science and Technology, Osaka University, 
1-5 Yamadaoka, Suita, Osaka 565-0871, Japan}
\email{kyasui@ist.osaka-u.ac.jp}
\date{September 14, 2020. \textit{Revised}: December 22, 2020}
\subjclass[2010]{Primary: 57R55. Secondary: 57R65}
\keywords{4-manifolds; Bauer--Furuta invariants; handle decompositions}

\begin{document}

\begin{abstract} We prove that, under a simple condition on the cohomology ring, every closed 4-manifold has mod 2 Seiberg--Witten simple type. This result shows that there exists a large class of topological 4-manifolds such that all smooth structures have mod 2 simple type, and yet some have non-vanishing (mod 2) Seiberg--Witten invariants. As corollaries, we obtain adjunction inequalities and show that, under a mild topological condition, every geometrically simply connected closed 4-manifold has the vanishing mod 2 Seiberg--Witten invariant for at least one orientation.    
\end{abstract}

\maketitle
\section{Introduction}\label{sec:introduction}The Seiberg--Witten invariant~\cite{Wi94} of a smooth 4-manifold has played a significant role in the study of 4-manifolds over the past 25 years, and has produced many striking applications to low dimensional topology. Although the invariants have been computed for various 4-manifolds, in general it still seems out of reach to compute. The simple type conjecture, posed in 1990s, states a fundamental constraint on the invariant (e.g.\ \cite[Conjecture 1.6.2]{KM07}). 

\begin{conjecture}[Simple type conjecture]\label{conj:simpletype}Every closed, connected, oriented, smooth 4-manifold with $b_2^+>1$ has Seiberg--Witten simple type. 
\end{conjecture}
Here a closed, connected, oriented, smooth 4-manifold $X$ with $b_2^+>1$ is called of \textit{Seiberg--Witten simple type} if the (integer valued) 
Seiberg--Witten invariant $SW_X(\mathfrak{s})$ (see \cite{Ni00, KM07}) of a spin$^c$ structure $\mathfrak{s}$ on $X$ is zero whenever the virtual dimension $d_X(\mathfrak{s})$ of the Seiberg--Witten moduli space for $\mathfrak{s}$ is non-zero. 
We note that $d_X(\mathfrak{s})=\frac{1}{4}(c_1(\mathfrak{s})^2-2\chi(X)-3\sigma(X))$, where $\chi$ and $\sigma$ respectively denote the Euler characteristic and the signature. Due to \cite{Wu48}, the conjecture is equivalent to the following: if $SW_X(\mathfrak{s})\neq 0$, then $c_1(\mathfrak{s})$ is the first Chern class of an almost complex structure on $X$. 

In the case where $b_2^+-b_1\equiv 0\pmod{2}$, $SW_X(\mathfrak{s})=0$ by the definition, and hence the conjecture is trivial. In the case where $b_2^+-b_1\equiv 1\pmod{2}$, the conjecture has been proved for many smooth 4-manifolds under smooth restrictions such as the existence of a symplectic structure (\cite{Tau96}). However, the conjecture remains open for general smooth structures on any topological 4-manifold. 

In this paper, we discuss the mod 2 version of the conjecture.  A closed, connected, oriented, smooth 4-manifold $X$ with $b_2^+>1$ will be called of \textit{mod 2 Seiberg--Witten simple type} if  $SW_X(\mathfrak{s})\equiv 0\pmod {2}$ whenever $d_X(\mathfrak{s})\neq 0$. Here we prove the mod 2 simple type conjecture under a simple condition on the cohomology ring. 

\begin{theorem}\label{intro:thm:cohomology_ring}Let $X$ be a closed, connected, oriented, smooth 4-manifold with $b_2^+-b_1>1$ and $b_2^+-b_1\equiv 3\pmod{4}$, and let $\{\delta_1, \delta_2, \dots, \delta_k\}$ be a generating set of $H^1(X;\mathbb{Z})$. If each cup product $\delta_i\cup \delta_j$ is either torsion or divisible by 2, then $X$ has mod 2 Seiberg--Witten simple type. 
\end{theorem}

This result gives the first examples of topological 4-manifolds such that all smooth structures are of mod 2 Seiberg--Witten simple type, and yet some have non-vanishing (mod 2) Seiberg--Witten invariants. As easily seen, this theorem provides a large class of such topological 4-manifolds. 

\begin{corollary}\label{intro:cor:sum}Let $X$ be a closed, connected, oriented, smooth 4-manifold with $b_2^+-b_1>1$ and $b_2^+-b_1\equiv 3\pmod{4}$. Suppose that the cohomology ring is isomorphic to that of a connected sum of (possibly more than two) closed oriented 4-manifolds, each summand of which satisfies either $b_1\leq 1$ or $b_2=0$. Then $X$ has mod 2 Seiberg--Witten simple type. 
\end{corollary}

\begin{corollary}\label{intro:cor:b_1<2}
Every closed, connected, oriented, smooth 4-manifold with $b_2^+>1$, $b_2^+-b_1\equiv 3\pmod{4}$ and $b_1\leq 1$ has mod 2 Seiberg--Witten simple type. 
\end{corollary}

We note that there are many 4-manifolds with non-vanishing mod 2 Seiberg--Witten invariants satisfying the assumption of Corollary~\ref{intro:cor:sum}. See, for example, \cite{GS99,BK07,Pa07,ABBKP10,To11,Yaz13,To14}. 
Also, the normal connected sum formula \cite[Corollary~3.3]{MST96} provides many 4-manifolds with non-vanishing mod 2 Seiberg--Witten invariants for which the mod 2 simple type conjecture is difficult to prove (without our results). 

\begin{remark}The $b_1=0$ case of Corollary~\ref{intro:cor:b_1<2} can be alternatively derived from results of Bauer and Furuta~\cite[Corollary 3.6 and Theorem 3.7]{BF04}. We note that the proofs of these results are homotopy theoretic and hence very different from ours. 
\end{remark}

We give simple applications of these results. We first discuss adjunction inequalities. A second cohomology class $K$ of a 4-manifold $X$ will be called a mod 2 Seiberg--Witten basic class if there exists a spin$^c$ structure $\mathfrak{s}$ on $X$ satisfying $K=c_1(\mathfrak{s})$ and $SW_X(\mathfrak{s})\not \equiv 0\pmod{2}$. Here we assume that every immersed sphere intersects itself only at transverse double points. Due to the generalized adjunction formula of Fintushel and Stern~\cite{FS95}, Theorem~\ref{intro:thm:cohomology_ring} implies the following adjunction inequality for immersed spheres. 

\begin{theorem}\label{intro:thm:immersed}Let $X$ be a closed, connected, oriented, smooth 4-manifold satisfying the assumption of Theorem~\ref{intro:thm:cohomology_ring}. Suppose that a second homology class $\alpha$ is represented by an immersed sphere having exactly $p_+$ positive double points and $p_-$ negative double points. If $p_+>0$ and $\alpha\cdot \alpha<0$, then any mod 2 Seiberg--Witten basic class $K$ satisfies 
\begin{equation*}
\left|\langle K, \alpha \rangle\right|+\alpha\cdot \alpha\leq 2p_+-2.
\end{equation*}
\end{theorem}

We note that this theorem holds for any 4-manifold satisfying the assumption of Corollary~\ref{intro:cor:sum}. For embedded surfaces, Corollary~\ref{intro:cor:b_1<2} implies the following adjunction inequality due to the generalized adjunction formula of Ozsv\'{a}th and Szab\'{o}~\cite{OzSz00Ann}. 

\begin{theorem}\label{intro:thm:embedded}Let $X$ be a closed, connected, oriented, smooth 4-manifold with $b_2^+>1$, $b_2^+-b_1\equiv 3\pmod{4}$ and $b_1\leq 1$. Suppose that a second homology class $\alpha$ is represented by a smoothly embedded, closed, oriented surface of genus $g$. If $g>0$ and $\alpha\cdot \alpha<0$, then any mod 2 Seiberg--Witten basic class $K$ satisfies 
\begin{equation*}
\left|\langle K, \alpha \rangle\right|+\alpha\cdot \alpha\leq 2g-2.
\end{equation*}
\end{theorem}

We next discuss the following conjecture, which states that the choice of an orientation of a 4-manifold imposes a strong constraint on the Seiberg--Witten invariant. 

\begin{conjecture}[cf.\ Kotschick~\cite{Ko92}, see \cite{Dra97}]\label{intro:conj:orientation}
Every simply connected, closed, oriented, smooth 4-manifold with $b_2^+>1$ and $b_2^->1$ has the vanishing Seiberg--Witten invariant for at least one orientation. 
\end{conjecture}
Kotschick~\cite{Ko97} proved this conjecture for a large class of complex surfaces (see also \cite{Dra97}). 
We remark that this conjecture has counterexamples if we remove the simply connected condition (e.g.\ the 4-torus). To state our result, let us recall that a compact, connected, smooth manifold is called \textit{geometrically simply connected} if it admits a handle decomposition without 1-handles. We note that a geometrically simply connected manifold is simply connected.  Also, we say that a 4-manifold $X$ has the \textit{vanishing mod 2 Seiberg--Witten invariant} if $SW_X(\mathfrak{s})\equiv 0\pmod {2}$ for any spin$^c$ structure $\mathfrak{s}$ on $X$. In \cite{Y19}, the third author showed that every geometrically simply connected, closed 4-manifold with $b_2^+\not\equiv 1$ and $b_2^-\not\equiv 1\pmod{4}$ admits no symplectic structure for at least one orientation. Improving  this result, Corollary~\ref{intro:cor:b_1<2} implies the mod 2 version of Conjecture~\ref{intro:conj:orientation} under a mild condition. 

\begin{theorem}\label{intro:thm:orientation}Every geometrically simply connected, closed, oriented, smooth 4-manifold with $b_2^+\not\equiv 1$ and $b_2^-\not\equiv 1\pmod{4}$ has the vanishing mod 2 Seiberg--Witten invariant for at least one orientation. 
\end{theorem}
We note that many simply connected, closed 4-manifolds including a large class of complex surfaces are geometrically simply connected (see \cite{GS99,Y19}). If this theorem does not hold without the condition ``geometrically'', then this theorem guarantees the existence of a counterexample to a long-standing open problem whether every simply connected, closed, smooth 4-manifold is geometrically simply connected (\cite[Problem 4.18]{Kir78}). For background on this problem, we refer to \cite{Y19}. In fact, we prove this theorem under a more general condition, which holds for many 4-manifolds including non-simply connected ones. Furthermore, this condition is much easier to verify. See Theorem~\ref{proof:thm:orientation}. 
\section{Proofs}
\subsection{Mod 2 Seiberg--Witten simple type}For a closed, connected, oriented 4-manifold $X$, let $[X]$ denote the fundamental class of $X$. We first prove the following theorem. 
\begin{theorem}\label{proof:thm:spinc}Let $X$ be a closed, connected, oriented, smooth 4-manifold with $b_2^+-b_1>1$ and $b_2^+-b_1\equiv 3\pmod{4}$, and let $\{\delta_1, \delta_2, \dots, \delta_k\}$ be a generating set of $H^1(X;\mathbb{Z})$. Suppose that a spin$^c$ structure $\mathfrak{s}$ on $X$ satisfies the following conditions. 
\begin{itemize}
 \item $SW_X(\mathfrak{s})\equiv 1\pmod {2}$. 
 \item $\langle c_1(\mathfrak{s})\cup \delta_i\cup \delta_j, [X] \rangle\equiv 0\pmod {4}$ for any $i,j$.
\end{itemize}
Then $d_X(\mathfrak{s})= 0$. 
\end{theorem}

\begin{proof}Suppose, to the contrary, that $d_X(\mathfrak{s})\neq 0$. Put $K=c_1(\mathfrak{s})$. Due to the assumption $SW_X(\mathfrak{s})\neq 0$, it follows from the definition of $SW_X(\mathfrak{s})$ that $d_X(\mathfrak{s})=2n$ for some positive integer $n$ (e.g.\ \cite[Section 2.3]{Ni00}). Let $X_n$ be the 4-manifold $X\#n\overline{\mathbb{C}\mathbb{P}^2}$, and let $K_n$ be its second cohomology class defined by $K_n=K+3E_1+3E_2+\dots+3E_n$, where each $E_i$ denotes the Poincar\'{e} dual of the second homology class $e_i$ of the $i$-th $\overline{\mathbb{C}\mathbb{P}^2}$ represented by the exceptional sphere. By the blow-up formula (\cite{FS95,Ni00}), we have a spin$^c$ structure $\mathfrak{s}_n$ on $X_n$ satisfying $c_1(\mathfrak{s}_n)=K_n$ and $SW_{X_n}(\mathfrak{s}_n)=SW_{X}(\mathfrak{s})\equiv 1\pmod {2}$. We note that $d_{X_n}(\mathfrak{s}_n)=0$.  Hence, we see that $(X_n, \mathfrak{s}_n)$ is BF admissible in the sense of Ishida and Sasahira~\cite[Definition 2]{IS15}, due to the assumption on $(X, \mathfrak{s})$. (We remark that the ``mod 2'' non-vanishing condition is a part of the definition of BF admissibility.) One can also check that $(K3, \mathfrak{t})$ is BF admissible, where $(K3, \mathfrak{t})$ denotes the $K3$ surface equipped with a spin$^c$ structure $\mathfrak{t}$ with $c_1(\mathfrak{t})=0$. By a result of Ishida and Sasahira \cite[Theorem A]{IS15} (see also \cite[Theorem 23 and Proposition 14]{IS15}) on the Bauer--Furuta invariant~\cite{BF04}, we see that $K_n$ is a Bauer--Furuta basic class of $Z:=X_n\#K3$, and thus a monopole class of $Z$ due to \cite[Proposition~6]{IL03}. 

Now let $\alpha$ be a second homology class of the $K3$ surface represented by a smoothly embedded, closed, oriented surface of genus $g>1$ satisfying $\alpha\cdot \alpha=2g-2$. As easily seen, there are many examples of such $\alpha$ (e.g.\ \cite[Theorem 1.1]{Ham14}). We note that the class $\alpha-e_1$ of the 4-manifold $Z$ is represented by a closed surface of genus $g$ with non-negative self-intersection number. 
Applying the adjunction inequality of Kronheimer~\cite[p.\ 53]{Kr99} to $Z$, we obtain the inequality 
\begin{equation*}
\langle K_n, \alpha-e_1 \rangle +(\alpha-e_1)\cdot (\alpha-e_1) \leq 2g-2.
\end{equation*}
Since the left side is $3+(2g-3)=2g$, the above inequality gives a contradiction.
\end{proof}

\begin{remark}(1) The role of the $K3$ surface in this proof can be replaced by any closed,  connected, oriented, smooth 4-manifold $Y$ with $b_2^+\equiv 3\pmod{4}$ and $b_1=0$ satisfying the following conditions: (i) $Y$ has a mod 2 Seiberg--Witten basic class; (ii) $Y$ has a smoothly embedded, closed surface of genus $g>1$ satisfying $\alpha\cdot \alpha=2g-2$. This can be easily checked by using the adjunction inequality (\cite{OzSz00JDG}). We remark that there are many examples of such $Y$. \\
(2) We used a connected sum formula of the Bauer--Furuta invariant to obtain a restriction on the smooth structure of a connected summand. A similar idea was used by the third author \cite{Y19} to impose constraints on geometrically simply connected 4-manifolds and, more generally, on 4-manifolds admitting a non-torsion second homology class represented by a 2-handle neighborhood. We remark that the $b_1=0$ condition of \cite[Theorem 2.4]{Y19} can be relaxed to conditions similar to Theorem~\ref{intro:thm:cohomology_ring} (and hence Corollaries~\ref{intro:cor:sum} and \ref{intro:cor:b_1<2}) of this paper without changing the proof, except that the connected sum formula of \cite{IS15} is used instead of the formula of \cite{Bau04}. 
\end{remark}
\begin{proof}[Proof of Theorem~\ref{intro:thm:cohomology_ring}]We note that $(\delta_i\cup \delta_j)\cup (\delta_i\cup \delta_j)=0$ for any $i,j$ (see also the proof of Lemma~\ref{proof:lem:cup}). It is thus easy to see that every spin$^c$ structure $\mathfrak{s}$ on $X$ satisfies $\langle c_1(\mathfrak{s})\cup \delta_i\cup \delta_j, [X] \rangle\equiv 0\pmod {4}$ for any $i,j$, since $c_1(\mathfrak{s})$ is characteristic, and any $\delta_i\cup \delta_j$ is either torsion or divisible by 2. Hence Theorem~\ref{intro:thm:cohomology_ring} follows from Theorem~\ref{proof:thm:spinc}. 
\end{proof}

We here observe the lemma below to prove Corollaries~\ref{intro:cor:sum} and \ref{intro:cor:b_1<2}. 
\begin{lemma}\label{proof:lem:cup}Let $X$ be a closed, connected, oriented, smooth 4-manifold with $b_1\leq 1$. Then for any classes $\gamma, \delta$ of $H^1(X;\mathbb{Z})$, the cup product $\gamma\cup \delta$ is zero. 
\end{lemma}
\begin{proof} By the universal coefficient theorem, we see that $H^1(X;\mathbb{Z})$ has no torsion. Due to the assumption $b_1(X)\leq 1$, it suffices to prove $\gamma\cup \gamma=0$ for any class $\gamma$ of $H^1(X;\mathbb{Z})$. We note that the Poincar\'{e} dual $PD(\gamma)$ is represented by a closed oriented codimension one submanifold of $X$ having a trivial normal bundle. This implies that $PD(\gamma\cup \gamma)$ is represented by the empty set, showing $\gamma\cup \gamma=0$. 
\end{proof}
\begin{proof}[Proof of Corollaries~\ref{intro:cor:sum} and \ref{intro:cor:b_1<2}] We note that any 4-manifold with $b_2^+>1$, $b_2^+-b_1\equiv 3\pmod{4}$ and $b_1\leq 1$ satisfies $b_2^+-b_1>1$. Hence, by the above lemma, Corollary~\ref{intro:cor:b_1<2} follows from Theorem~\ref{intro:thm:cohomology_ring}. Therefore, we easily see that Corollary~\ref{intro:cor:sum} also follows from Theorem~\ref{intro:thm:cohomology_ring}. 
\end{proof}

\subsection{Adjunction inequalities}

\begin{proof}[Proof of Theorem~\ref{intro:thm:immersed}]Suppose, to the contrary, that $\left|\langle K, \alpha \rangle\right|+\alpha\cdot \alpha>2p_+-2$ for some $K=c_1(\mathfrak{s})$ and $\alpha$. Then the generalized adjunction formula of Fintushel and Stern~\cite[Theorem 1.3]{FS95} shows that $K'=K+2\epsilon PD(\alpha)$ is a mod 2 Seiberg--Witten basic class, where $\epsilon=\pm1$ is the sign of $\langle K, \alpha \rangle$. It is straightforward to see that $K'=c_1(\mathfrak{s}')$ satisfies 
\begin{equation*}
d_X(\mathfrak{s}') = d_X(\mathfrak{s})+\left|\langle K, \alpha \rangle\right|+\alpha\cdot \alpha> 0.
\end{equation*}
Since $X$ is of mod 2 Seiberg--Witten simple type due to Theorem~\ref{intro:thm:cohomology_ring}, this is a contradiction. 
\end{proof}

Ozsv{\'a}th and Szab{\'o}~\cite{OzSz00Ann} introduced the Seiberg--Witten invariant of the form  
\[
SW_{X,\mathfrak{s}}\colon \mathbb{A}(X) \to \mathbb{Z}
\]
for a spin$^c$ structure $\mathfrak{s}$,
where $\mathbb{A}(X) = \bigwedge H_1(X;\mathbb{Z})\otimes\mathbb{Z}[U]$,  $H_1(X;\mathbb{Z})$ has grading $1$ and $U$ is a degree $2$ generator (cf.\ \cite{Tau99}).
This function and the integer valued invariant have the relation 
\[
SW_{X,\mathfrak{s}}(U^{d_X(\mathfrak{s})/2}) = SW_X(\mathfrak{s})
\] 
when $d_X(\mathfrak{s})$ is non-negative and even.
Let $\Sigma$ be a smoothly embedded, closed, oriented surface of genus $g$ representing $\alpha$. 
For such  a surface $\Sigma$, they defined the class $\xi(\Sigma)\in \mathbb{A}(X)$ by
\[
\xi(\Sigma) = \prod_{i=1}^g(U-A_i\cdot B_i),
\]
where $\{A_i,B_i\}_{i=1}^g$ are the images in $H_1(X;\mathbb{Z})$ of a standard symplectic basis for $H_1(\Sigma;\mathbb{Z})$.
\begin{proof}[Proof of Theorem~\ref{intro:thm:embedded}]
Suppose $\left|\langle K, \alpha \rangle\right|+\alpha\cdot \alpha > 2g-2$.
Then, 
by \cite[Theorem 1.3]{OzSz00Ann}, the relation
\[
SW_{X, \mathfrak{s}+\epsilon\alpha}(\xi(\epsilon\Sigma)U^m) = SW_{X,\mathfrak{s}}(1)
\]
holds, 
where $\epsilon=\pm1$ is the sign of $\langle K,\alpha\rangle$ and $2m=|\langle K,\alpha\rangle|+\alpha\cdot\alpha-2g$.
Since $\xi(\Sigma)=U^g$ when $b_1(X)\leq 1$, we obtain
\[
SW_X(\mathfrak{s}+\epsilon\alpha) = SW_{X, \mathfrak{s}+\epsilon\alpha}(U^{m+g}) = SW_{X,\mathfrak{s}}(1)=SW_X(\mathfrak{s})\underset{(2)}{\equiv} 1.
\]
On the other hand, $d_X(\mathfrak{s}+\epsilon\alpha) = d_X(\mathfrak{s})+ \left|\langle K, \alpha \rangle\right|+\alpha\cdot \alpha>0$.
This contradicts  Corollary~\ref{intro:cor:b_1<2}.
\end{proof}
\begin{remark}
Conjecture \ref{conj:simpletype} was originally posed for $4$-manifolds with $b_1=0$ in relation to Witten's conjecture  \cite{Wi94,FL15} on the relationship between the Donaldson and Seiberg--Witten invariants,  and was later extended to the case of arbitrary $b_1$ in the literature.
Ozsv{\'a}th and Szab{\'o}~\cite{OzSz00Ann} gave a stronger version of the simple type condition.
They call $X$  of simple type when the function $SW_{X,\mathfrak{s}}$ is identically zero if $d_X(\mathfrak{s})\neq 0$.
Taubes \cite[Proof of Proposition 2.2]{Tau99} proved that every closed symplectic $4$-manifold $X$ with $b_2^+(X)>1$ is of simple type in this strong sense (see also \cite[Remark 3.3]{OzSz00Ann}). 
However, in contrast to the case of (ordinary) simple type, there are many 4-manifolds which are not of strong simple type.
For instance, it follows from the surgery formula of Ozsv{\'a}th and Szab{\'o}~\cite[Proposition 2.2]{OzSz00JDG}  that a connected sum $X\#(S^1 \times S^3)$ is such an example when $X$ has a non-vanishing integer Seiberg--Witten invariant.
\end{remark}

\subsection{A vanishing theorem for mod 2 Seiberg--Witten invariants}
We recall a definition and a lemma given in \cite{Y19}. 
\begin{definition}
Let $\alpha$ be a second homology class of a smooth 4-manifold $X$. We say that $\alpha$ is represented by a \textit{2-handle neighborhood}, if $X$ has a codimension zero submanifold $W$ satisfying the following conditions. 
\begin{itemize}
 \item The submanifold $W$ is diffeomorphic to a 4-manifold obtained from the 4-ball by attaching a single 2-handle. (This submanifold will be called a 2-handle neighborhood.) 
 \item $\alpha$ is the image of a generator of $H_2(W;\mathbb{Z})\cong \mathbb{Z}$ by the inclusion induced homomorphism $H_2(W;\mathbb{Z})\to H_2(X;\mathbb{Z})$. 
\end{itemize}
\end{definition}
\begin{lemma}[{\cite[Lemma 3.1]{Y19}}]\label{proof:lem:geometrically}Every second homology class of a geometrically simply connected, compact, smooth 4-manifold is represented by a 2-handle neighborhood. 
\end{lemma}

For an oriented 4-manifold $X$, let $\overline{X}$ denote the 4-manifold $X$ equipped with the reverse orientation.  We show the following vanishing theorem for mod 2 Seiberg--Witten invariants. 
\begin{theorem}\label{proof:thm:orientation}Let $X$ be a closed, connected, oriented, smooth 4-manifold satisfying $b_2^+-b_1>1$, $b_2^--b_1>1$, $b_2^+-b_1\not\equiv 1$ and $b_2^--b_1\not\equiv 1\pmod{4}$, and let $\{\delta_1, \delta_2, \dots, \delta_k\}$ be a generating set of $H^1(X;\mathbb{Z})$. Suppose that each cup product $\delta_i\cup \delta_j$ is either torsion or divisible by 2. If $X$ admits a non-torsion second homology class represented by a 2-handle neighborhood, then at least one of $X$ and $\overline{X}$ has the vanishing mod 2 Seiberg--Witten invariant. 
\end{theorem}
We note that the existence of a non-torsion second homology class represented by a 2-handle neighborhood is much easier to verify than the geometrically simply connected condition, since it is often not necessary to decompose an entire 4-manifold into a handlebody. Indeed, many closed 4-manifolds including non-simply connected ones admit such second homology classes. See \cite[Section 3]{Y19} for more background on such 4-manifolds. 

The proof of this theorem relies on the following result. 
\begin{theorem}[\cite{Y19}]\label{proof:thm:orientation_d=0}Let $X$ be a 4-manifold satisfying the assumption of Theorem~\ref{proof:thm:orientation}. Then at least one of the following properties holds. 
\begin{itemize}
 \item Every spin$^c$ structure $\mathfrak{s}$ with $d_X(\mathfrak{s})=0$ satisfies $SW_X(\mathfrak{s})\equiv 0\pmod {2}$. 
 \item Every spin$^c$ structure $\mathfrak{s}$ with $d_{\overline{X}}(\mathfrak{s})=0$ satisfies $SW_{\overline{X}}(\mathfrak{s})\equiv 0\pmod {2}$. 
\end{itemize}
\end{theorem}
This theorem is implicit in the proof of \cite[Theorem~2.4]{Y19}, which states that any 4-manifold with $b_1=0$ satisfying the assumption of Theorem~\ref{proof:thm:orientation} admits no symplectic structure for at least one orientation. The proof for the $b_1=0$ case of Theorem~\ref{proof:thm:orientation_d=0} is identical with the proof of \cite[Theorem~2.4]{Y19}, and the proof for the general case also is identical, except that the connected sum formula of \cite{IS15} is used instead of the formula of \cite{Bau04}. 

\begin{proof}[Proof of Theorem~\ref{proof:thm:orientation}]
This is straightforward from Theorems~\ref{proof:thm:orientation_d=0} and \ref{intro:thm:cohomology_ring}. 
\end{proof}
\begin{proof}[Proof of Theorem~\ref{intro:thm:orientation}]
This is straightforward from Theorem~\ref{proof:thm:orientation} and Lemma~\ref{proof:lem:geometrically}. 
\end{proof}

\subsection*{Acknowledgements}The authors would like to thank Hirofumi Sasahira for helpful comments. Kato was partially supported by JSPS KAKENHI Grant Numbers 17H02841 and 17H06461. Nakamura was partially supported by JSPS KAKENHI Grant Number 19K03506. Yasui was partially supported by JSPS KAKENHI Grant Numbers 17K05220, 19H01788, and 19K03491. 

\end{document}